\documentclass[12pt,a4paper]{amsart}
\pdfoutput=1

\usepackage{amsmath,amssymb,amsthm}  
\usepackage{mathtools}

\usepackage{xcolor} 	
\usepackage{hyperref}
\hypersetup{
	colorlinks,
    linkcolor={red!60!black},
    citecolor={green!60!black},
    urlcolor={blue!60!black},
}
\usepackage[abbrev, msc-links]{amsrefs} 
\usepackage{thmtools, thm-restate}

\usepackage[utf8]{inputenc}
\usepackage[T1]{fontenc}
\usepackage{lmodern}
\usepackage[babel]{microtype}
\usepackage[english]{babel}

\usepackage{geometry}
\usepackage{enumitem}

\theoremstyle{plain}
\newtheorem{thm}{Theorem}[section]
\newtheorem*{thm*}{Theorem}
\newtheorem{fact}[thm]{Fact}
\newtheorem{prop}[thm]{Proposition}
\newtheorem{defn}[thm]{Definition}

\newtheorem{cor}[thm]{Corollary}
\newtheorem{cond}[thm]{Condition}
\newtheorem{lem}[thm]{Lemma}

\newtheorem{obs}[thm]{Observation}

\newtheorem{quest}[thm]{Question}
\newtheorem{conj}[thm]{Conjecture}

\theoremstyle{definition}

\title{Intersection of a partitional and a general infinite matroid}

\author{Attila Jo\'{o}}
\thanks{The author would like to thank the generous support of the Alexander 
von Humboldt Foundation and NKFIH 
OTKA-129211}
\address{Attila Jo\'{o},
University of Hamburg, Department of Mathematics, Bundesstra{\ss}e 55 (Geomatikum), 20146 Hamburg, Germany}
\email{attila.joo@uni-hamburg.de}
\address{Attila Jo\'{o},
Alfr\'{e}d R\'{e}nyi Institute of Mathematics, Set theory and general topology research division, 13-15 Re\'{a}ltanoda St., 
Budapest, Hungary}
\email{jooattila@renyi.hu}

\begin{document}

\begin{abstract}
   Let $ E $ be a possibly infinite set and let $ M $ and $ N $ be matroids defined on $ E $. We say that the pair $ \{ M,N \} $ has 
   the Intersection property if $ M $ and $ N $ share an independent set $ I $ admitting a bipartition $ I_M\sqcup I_N $ such that
   $ \mathsf{span}_M(I_M)\cup \mathsf{span}_N(I_N)=E $. The Matroid Intersection Conjecture of Nash-Williams says that 
   every matroid pair has the Intersection property.
   
  The conjecture is known and easy to prove in the case when one of the matroids is uniform and it was shown by  Bowler and 
   Carmesin that the conjecture is implied by its special case where one of the matroids is a direct sum of uniform matroids, i.e., is
   a partitional matroid. We show that if $ M $ is an arbitrary matroid and $ N $ is the direct 
   sum of finitely many uniform matroids, then $ \{ M, N \} $ has the Intersection property.
\end{abstract}

\maketitle
\section{Introduction}
\subsection{Infinite matroids and the Matroid Intersection Conjecture}
 Some of the motivating examples of matroids are vector-systems with  linear independence and graphs with graph 
 theoretic cycles as circuits. Both types of structures can be infinite 
 in which case the resulting matroid is infinite as well.  An axiomatization of matroids (in the 
 language of circuits) that allows 
  infinite ground sets  can be obtained from the axiomatization of finite matroids  in a natural 
 way:  $ \mathcal{C} $ is the set of the circuits of a finitary matroid\footnote{A matroid with only finite circuits called finitary.}  
 if  $ \mathcal{C} $ is a family of finite 
 nonempty  pairwise $ \subseteq 
 $-incomparable subsets  of a 
 possible infinite set $ E $ satisfying the Circuit elimination axiom.\footnote{If $ C_0, C_1\in \mathcal{C} $ are distinct and $ 
 e\in C_0\cap C_1 $, then $ \exists C_2\in \mathcal{C} $ with $ C_2\subseteq C_0\cup C_1-e $.} Working with this 
 definition, Nash-Williams proposed  his Matroid Intersection Conjecture  
 \cite{aharoni1998intersection}  which has been the most important open problem in 
infinite matroid theory for decades. It generalizes the Matroid Intersection Theorem of Edmonds 
\cite{edmonds2003submodular} to infinite 
matroids capturing  the combinatorial structure corresponding to the largest common independent sets instead of dealing with 
infinite 
quantities (cardinality usually turns out to be an overly rough measure for problems in infinite combinatorics). Adopting a 
terminology of Bowler and Carmesin,  for a pair 
$ \{ M,N \} $ of matroids defined on the same edge set $ E $ we say  that it has the Intersection property if $ M $ and $ N $ has a 
common  independent set $ I $ admitting a bipartition $ I_M\sqcup I_N $ such that $ \mathsf{span}_M(I_M)\cup 
\mathsf{span}_N(I_N)=E $.

\begin{conj}[Matroid Intersection Conjecture, \cite{aharoni1998intersection}*{Conjecture 1.2}]\label{MIC}
Every pair $ \{ M,N \} $ of  matroids defined on the same (potentially infinite) ground set has the Intersection property.
\end{conj}
The concept of infinite matroids described in the  first paragraph was not entirely satisfying for the experts in matroid 
theory. Indeed,  matroids may 
fail to have a dual although duality is a fundamental
phenomenon of the finite theory. 
More precisely, their formally defined duals fail to be a matroid. For example in an infinite connected graph the bonds are 
supposed 
to be 
the circuits of the 
dual of its cycle matroid because they are the minimal edge sets meeting with every base. But these bonds are usually not even 
finite.
Rado asked in 1966 for a more general 
definition of infinite matroids that allows infinite circuits (see 
\cite{rado1966abstract}*{Problem P531}).  Among other 
attempts Higgs \cite{higgs1969matroids}
introduced  a class of structures he 
called 
``B-matroids'' that solves Rado's problem. 
Oxley gave an axiomatization of B-matroids and showed that  this is the
 broadest class of structures satisfying some natural `matroid-like' axioms (namely \ref{item axiom1}-\ref{item axiom3} below) 
 for 
 which the usual definition of dual and minors are meaningful and results in a matroid   (see 
 \cite{oxley1978infinite} and \cite{oxley1992infinite}). 
 Despite these 
 discoveries 
 of Higgs and Oxley, the systematic investigation of B-matroids started only around 2010 when Bruhn, Diestel, Kriesell, 
 Pendavingh and Wollan  
found a 
set of cryptographic axioms for them, generalising the usual independent set-, bases-, circuit-, closure- and 
rank-axioms for finite matroids (see \cite{bruhn2013axioms}). They also showed that several well-known facts of the theory of 
finite 
matroids are preserved.  
Their axiomatization in the language of independent sets is the following:

 $ M=(E, \mathcal{I}) $ is a B-matroid (or simply matroid)  if $ \mathcal{I}\subseteq \mathcal{P}(E) $ with
\begin{enumerate}
[label=(\roman*)]
\item\label{item axiom1} $ \varnothing\in  \mathcal{I} $;
\item\label{item axiom2} $ \mathcal{I} $ is downward closed;
\item\label{item axiom3} For every $ I,J\in \mathcal{I} $ where  $J $ is $ \subseteq $-maximal in $ \mathcal{I} $ but $ I $ is 
not, there 
exists an
 $  e\in 
J\setminus I $ such that
$ I+e\in \mathcal{I} $;
\item\label{item axiom4} For every $ X\subseteq E $, any $ I\in \mathcal{I}\cap 
\mathcal{P}(X)  $ can be extended to a $ \subseteq $-maximal element of 
$ \mathcal{I}\cap \mathcal{P}(X) $.
\end{enumerate}

After this success of Rado's program the name `Matroid Intersection Conjecture' gained a new interpretation by 
applying 
the definition above  instead of the more restrictive infinite matroid concept that allows only finite circuits. Although several 
partial results have been obtained about the Matroid Intersection Conjecture (see 
\cite{aharoni1998intersection}, \cite{bowler2015matroid}, \cite{ghaderi2017}, \cite{aigner2018intersection},  
\cite{bowler2020almost}, \cite{joo2020MIC}), even its
original  version considering only finitary matroids remains wide open.

\subsection{Uniform matroids}
Let $ E $ be a (possibly infinite) set of size at least $ n\in \mathbb{N} $. Then the  subsets of $ E $ of size at most $ n $ are the 
independent sets of a matroid $ U_{E,n} $  which is called the $ n $-uniform 
matroid on $ E $.  Talking about $ \kappa $-uniform matroids for an infinite cardinal $ \kappa $ is not meaningful in general. 
Indeed,  if $ 
\left|E\right|>\kappa $, then the set of subsets of $ E $ of size at most $ \kappa $  fails to be the set of 
the independent sets of a matroid
 because there is no maximal among these sets which is a violation of  axiom 
 \ref{item axiom4} for $ I=\varnothing $ and $ X=E $.  The ``right'' concept of (potentially infinite)  uniform matroids 
was found by Bowler and Geschke:
\begin{defn}[\cite{bowler2016self}*{Definition 2}]\label{def} 
 A matroid $ (E,\mathcal{I}) $ is uniform if for every $ I\in \mathcal{I},\ e\in I $ and $ f\in E\setminus I $: $ I-e+f\in 
 \mathcal{I} $. 
\end{defn}
The matroids $ \{ U_{E,n}: n\in \mathbb{N} \} $ and their duals are uniform according to Definition \ref{def}.
 It is not too hard to show (see Proposition \ref{prop: fintaryUnif}) that if a uniform matroid is finitary 
then it is either free or $ n $-uniform for some $ n\in \mathbb{N} $. A 
natural question if there are uniform matroids that are neither finitary nor cofinitary\footnote{A matroid is cofinitary if its dual is 
finitary.}. The positive answer (under 
some set theoretic assumptions) was shown  by Bowler and Geschke 
in~\cite{bowler2016self}. They developed an efficient method to build uniform matroids while having a large degree of freedom 
during the construction.  Their technique turned out 
to be a powerful tool to prove relative consistency results in infinite matroid theory. Higgs showed in 
\cite{higgs1969equicardinality} that under the 
Generalized 
Continuum Hypotheses for every fixed matroid $ M $ the bases of $ M $ have the same cardinality and asked if it is provable 
already in set theory ZFC. The negative answer was shown by Bowler and Geschke which means that the question about the 
equicardinality of the bases  is undecidable in ZFC. Another example for an application of uniform matroids is 
corresponding to the following problem. Suppose that $ M $ and $ N $ are matroids on a common 
ground set and $ M $ has a base contained in a base of $ N $ and vice versa. Do $ M $ and $ N $ necessarily share a base? The 
positive answer is known for the special case where each of the matroids is either finitary or cofinitary 
\cite{erde2019base}*{Corollary 
1.4}. The unprovability of the general case \cite{erde2019base}*{Theorem 5.1} was demonstrated via the 
 method of Bowler and Geschke.

A question of undoubtedly great importance in the field is if the existence of uniform matroids that 
are neither finitary nor cofinitary is provable in ZFC alone. The existence of such a matroid  on a $ \kappa $-sized ground set 
is known to be equivalent with the positive answer for the following set theoretic question:
\begin{quest}\label{ques: uniexis}
Let $ \mathcal{P}(\kappa)/\mathrm{fin} $ be the power set of $ \kappa $ factorized by the equivalence relation 
$ \sim $ where $ X\sim Y $ iff the symmetric different $ X \vartriangle Y $ is finite. For $ [X], [Y]\in 
\mathcal{P}(\kappa)/\mathrm{fin} $ we let $ 
[X]\subseteq^{*}  [Y]$ iff $ X\setminus Y $ is finite. Is there a family 
$ \mathcal{A}\subseteq\mathcal{P}(\kappa)/\mathrm{fin} $  satisfying the 
following conditions?
\begin{enumerate}
\item\label{item: antichain} $ \mathcal{A} $ is an antichain, i.e., it consists of $ \subseteq^{*} $-incomparable elements;
\item\label{item: pairs} $ \mathcal{A} $ is line-dense in $ \mathcal{P}(\kappa)/\mathrm{fin} $, i.e., for every $ [X], [Y]\in 
\mathcal{P}(\kappa)/\mathrm{fin} $ with $ [X]\subseteq^{*}  [Y] $, there is an $ [A]\in 
\mathcal{A} $ such that at least one of the following holds
\begin{itemize}
\item $ [A]\subseteq^{*}  [X] $
\item $ [X]\subseteq^{*} [A] \subseteq^{*} [Y]  $
\item  $  [Y] \subseteq^{*} [A] $;
\end{itemize}
\item $ \mathcal{A} $ is non-trivial, i.e., $ \mathcal{A}$ is neither $ \{ [\varnothing] \}$ nor $ \{ [\kappa] \} $.
\end{enumerate}
\end{quest}

\subsection{Our main result}
The special case of Conjecture \ref{MIC} where $ N $ is assumed to be uniform is well-known and easy to prove
while the special case where $ N $ is the direct sum of uniform matroids (known as partitional matroids) is already 
equivalent 
with the conjecture itself (see \cite{bowler2015matroid}*{Corollary 3.9 (b)}).   Our main result is deciding the problem  between 
these two extremal cases:

\begin{restatable}{thm}{main}\label{thm:main}
Let $ M $ and $ N $ be matroids on  $ E $ such that  $ N=\bigoplus_{i<n}U_i $ where $ n\in 
\mathbb{N}$ and  $ U_i $ is a uniform 
matroid on  $ E_i $  with $ E_i\cap E_j=\varnothing $ for $ 0\leq i<j<n $. Then $ \{ M, N \} $ has the 
Intersection property.
\end{restatable}

 Having the unprovability results mentioned in the previous 
subsection in mind, we phrase the following question in a careful way:

\begin{quest}
Is it provable in ZFC that a pair $ \{ M,N \} $ of partitional matroids on a common countable ground set always has the 
Intersection property?
\end{quest}

\section{Notation and Preliminaries}
We apply the standard set theoretic convention that natural numbers are identified with the set of smaller natural numbers, i.e.,
$ n=\{ 0,\dots, n-1 \} $. For ease of presentation, when there is no chance of misunderstanding, we write simply $ X-y+z $ instead 
of $ X\setminus \{ y \}\cup \{ z \} $. The defined terms and symbols are highlighted with italic and  bold respectively for 
convenience.

\subsection{General matroidal terms}
A pair ${M=(E,\mathcal{I})}$ is a \emph{matroid} if ${\mathcal{I} \subseteq \mathcal{P}(E)}$ satisfies  the axioms 
\ref{item axiom1}-\ref{item axiom4}. A matroid is \emph{trivial} if $ E=\varnothing $ and \emph{free} if $ 
\mathcal{I}=\mathcal{P}(E) $.
The sets in~$\mathcal{I}$ are called \emph{independent} while the elements of $ \mathcal{P}(E) \setminus \mathcal{I} $ are 
the \emph{dependent} sets. The maximal independent sets are the \emph{bases}.

\begin{fact}[\cite{bruhn2013axioms}*{Lemma 3.7}]\label{fact: equicBase}
If matroid $ M $ has a finite base, then every base of $ M $ has the same size.\footnote{The equicardinality of bases can be 
proved for some larger classes of matroids but as already mentioned  is  independent of set theory ZFC considering the class of all 
matroids.}
\end{fact}
Every dependent set contains a minimal dependent set, these are called 
\emph{circuits}. For an  ${X \subseteq E}$, ${\boldsymbol{M  \upharpoonright X} :=(X,\mathcal{I} 
\cap \mathcal{P}(X))}$ is a matroid referred as the \emph{restriction} of~$M$ to~$X$.
We write ${\boldsymbol{M - X}}$ for $ M  \upharpoonright (E\setminus X) $  and call it the minor obtained by the 
\emph{deletion} of~$X$. 
Let $ B_X $ be a maximal independent subset of $ X $. The \emph{contraction} $ \boldsymbol{M/X} $ of $ X $ in $ M $ is the  
matroid on $ E\setminus X $ where 
$ I\subseteq E\setminus X $ is independent iff $ I\cup B_X $ is independent in $ M $. One can show that the definition does not 
depend on the choice of $ B_X $.  
Contraction and deletion commute, i.e., for 
disjoint 
$ X,Y\subseteq E $, we have $ (M/X)-Y=(M-Y)/X $. Matroids of this form are the  \emph{minors} of~$M$.
The set $ \mathsf{span}_{M}(X) $  of edges spanned by $ X $ in $ M $ consists of the edges in $ X $ and of those $ e\in 
E\setminus X $ for which $ \{ e \} $ is dependent  in $ M/X $. The operator $ \mathsf{span_M} $ is idempotent  and 
hence (since extensivity and monotonicity are straightforward from the definition)  is a closure operator.  An $ S\subseteq E $ is 
\emph{spanning} in $ M $ if $ 
\mathsf{span}_M(S)=E $.  Bases are exactly the independent spanning sets.

Let  $\Theta $ be some index set. For $ i\in \Theta $, let  $ M_i  $ be a matroid on $ E_i $ such that $ E_i\cap E_j=\varnothing $ 
for $ i\neq j $. Then the 
\emph{direct sum} 
$ \boldsymbol{\bigoplus_{i\in \Theta}M_i} $ is the matroid defined on $ E:=\bigcup_{i\in \Theta}E_i$ where $ I\subseteq E$  is 
independent in $  \bigoplus_{i\in 
\Theta}M_i$ iff $ I\cap E_i $ is independent in $ M_i $ for every $ i\in  \Theta $. For a detailed introduction to the theory of 
infinite matroids we refer to  \cite{nathanhabil}.

\subsection{Basic facts about uniform matroids}

We will make use of the following characterisation of uniform matroids (recall Definition \ref{def}):
\begin{prop}\label{prop:uniform}
A matroid $ U $ on $ E $ is uniform if and only if  for every $ F\subseteq E $, $ F $ is either independent or  spanning in $ U $.
\end{prop}
\begin{proof}
Suppose that $ U $ is uniform and let $ F\subseteq E $ be arbitrary.  
We take   be a maximal independent subset $ B $ of $ F $. We may assume that $ B\subsetneq F $ since otherwise $ F=B $ is 
independent and we are done.  Suppose for a contradiction that $ B $ is not a base of $ U $. Then 
there is some $ e\in E\setminus B $ such that $ B+e $ is independent. By the choice of $ B $ we know that $ e\in E\setminus F $. 
Let $ f\in F\setminus B $ be arbitrary. By  uniformity (Definition \ref{def}), $ B+f $ must be independent, since we can obtain it 
from $ B+e $ by deleting $ e $ and adding $ f $. But then $ B+f $ is an independent subset of $ F $ contradicting the maximality 
of  $ B $. Thus $ B $ is a base and therefore $ F $ is spanning.

Conversely, assume that every $ F\subseteq E $ is either independent or spanning  in the matroid $ U=(E, \mathcal{I}) $. 
Suppose for a 
contradiction that there are $ I\in \mathcal{I},\ e\in I $ and $ f\in E\setminus I $ such that $ I-e+f\notin \mathcal{I} $. Then by 
assumption $ I-e+f $ is spanning, thus contains some  base $ B $.  We must have $ f\in B $ since otherwise $ B\subseteq 
I-e\subsetneq I $  contradicts the fact that $ I $ is independent. We also know that $ B\subsetneq I-e+f $ because $  I-e+f $ 
is dependent, therefore we can fix some $ h\in (I-e+f)\setminus B $. But then
$ \{ f \} $ is a base of $ U/(B-f) $ and $ \{ e, h \}\in \mathcal{I}_{U/(B-f)} $ because $ B-f+e+h\subseteq I\in \mathcal{I} $ 
which contradicts Fact \ref{fact: equicBase} after 
extending  $ \{ e, h 
\} $ to a base of $ U/(B-f) $.
\end{proof}
\begin{cor}
The class of uniform matroids is closed under duality and under taking minors.
\end{cor}

\begin{cor}\label{cor: MbaseBcontains}
If $ M $ and $ U $ are matroids on $ E $ and $ U $ is uniform, then  there is either an $ M $-independent base of $ U $ 
or  an $ U 
$-independent base of $ M $.
\end{cor}
\begin{proof}
Let $ B $ be an arbitrary base of $ M $. If $ B $ is independent in $ U $, then we are done. Otherwise  $ B $ must be spanning in 
$ U $, thus $ B $ contains a base of $ U $ which is then independent in $ M $.
\end{proof}
\begin{prop}\label{prop: fintaryUnif}
If the uniform matroid $ U=(E, \mathcal{I}) $ is finitary, then $ U $ is either free or of the form  $ U_{E,n} $ for suitable $ n\in 
\mathbb{N} $.
\end{prop}
\begin{proof}
If every finite set is independent in $ U $ then it has no circuits (because every circuit of $ U $ is finite by assumption), thus $ U $ 
is free. If $ F\subseteq E $ is a finite dependent set, then $ F $ is spanning  by Proposition \ref{prop:uniform}, thus there is a 
finite 
base $ B\subseteq F $. 
But then Fact \ref{fact: equicBase} ensures that every base has size $ \left|B\right|=:n $. If an $ H\subseteq E $  of size $ n $ 
were not a base, then by applying Proposition \ref{prop:uniform} there would be a base $ B' $ such that either $ H\subsetneq B' $ 
or $ 
B'\subsetneq H $ holds in both of which cases $ \left|B'\right|\neq n$, a contradiction.  Thus every $ n $-element subset is a base 
and  
therefore $ U=U_{E,n} $.
\end{proof}
\begin{prop}\label{prop: spanUni}
Assume that $ U $ is a uniform matroid on $ E $, $ I $ is independent in $ U $ and $ I $ spans $ e\in E\setminus I $ in $ U $. Then 
$ I $ must be a base of $ U $.
\end{prop}
\begin{proof}
By assumption $ I+e $ is dependent and therefore also spanning by Proposition \ref{prop:uniform}. But then $ I $ is spanning 
because it spans the spanning set $ I+e $. Since $ I $ is independent as well, it is a base.
\end{proof}
\section{Proof of the main result}
\subsection{Preparatory lemmas}
\begin{obs}\label{obs:baseContained}
 If $ M $ and $ N $ are arbitrary matroids on $ E $ and there is an $ M $-independent base $ B_N $ of  $ N $ ($ N $-independent 
 base $ 
 B_M $ of  $ M $), then $ B_N $ ($ B_M $)  witnesses  the 
 Intersection property of $ \{ M, N \} $ via the trivial bipartition $ \varnothing\sqcup B_{N} $ ($  B_M\sqcup \varnothing $).
 \end{obs}
From now on, let $ M $ and $ N $ be matroids on $ E $ such that  $ N=\bigoplus_{i<n}U_i $ where $ n\in 
 \mathbb{N} $ and  $ U_i $ is a uniform matroid on  $ E_i $  with $ E_i\cap E_j=\varnothing $ for $ 0\leq i<j<n $. 
The union of 
 an $ \subseteq $-increasing sequence of common independent sets of $ M $ and $ N $ is not a common independent set in 
 general because both $ M $ and $ N $ may have infinite circuits. Therefore Zorn's lemma cannot be used to extend common 
 independent sets to maximal ones as in the case of finitary matroids. Even so, such an extension is still always possible:
 
\begin{lem}\label{lem:extend}
Every common independent set $ I $ of $ M $ and $ N$ can be extended to a maximal common independent set.
\end{lem}

\begin{proof}
Let $ i<n $ be arbitrary.  By applying Corollary \ref{cor: MbaseBcontains} with $ M/I \upharpoonright (E_i\setminus I) $ and  $ 
U_i/(I\cap E_i) $ we obtain a common independent set $ B $ of these matroids  which is in addition a base in at least one of them. 
It 
follows directly from the construction that $ I\cup B\in \mathcal{I}_M\cap 
\mathcal{I}_N $, furthermore, $ 
I\cup B $ either $ M 
$-spans or $ N $-spans $ E_i $ depending on whose base was $ B $.  Iterating this with all indices yields a maximal 
common independent set.
\end{proof}
In the previous proof  we obtained a seemingly stronger property than maximality. Let us point out that 
it is actually equivalent:
\begin{lem}\label{lem:char maximal}
An $ I\in  \mathcal{I}_M\cap 
\mathcal{I}_N  $ is maximal in $ \mathcal{I}_M\cap 
\mathcal{I}_N $ if and only if every $ E_i $ is spanned by $ I $ in at least one of the matroids.
\end{lem}
\begin{proof}
The ``if'' direction is straightforward. To show the ``only if'' part let $ I $ be a maximal common independent set and let $ i<n $. 
On the 
one hand, every $ e\in E_i \setminus I $ is spanned by $ I $ in at least one of the matroids because $ I+e\notin  
\mathcal{I}_M\cap 
\mathcal{I}_N   $. On the 
other hand, either $ I\cap E_i $ is a base of $ U_i $ in which case $ E_i\subseteq \mathsf{span}_N(I) $, or $ I\cap E_i $ is not a 
base of $ U_i $ but then by Proposition \ref{prop: spanUni} $ I $ does not $ N $-span any edges from $ E_i\setminus I $ and 
hence we must have $ E_i\subseteq \mathsf{span}_M(I) $.
\end{proof}

\subsection{Proof of the main theorem}
Let us repeat here our main theorem for convenience:

\main*
We apply induction on $ n $. For $ n=0 $, both $ M $ and $ N $ are trivial and $ \varnothing $ witnesses the Intersection property 
of $ \{ M, N \} $ via $ 
  \varnothing\sqcup \varnothing $.
  
  Suppose now that  $ n\geq1 $. 
  Assume first that there is a nonempty $W\subseteq E $ which is the union of some of the sets $ E_i $  such that 
  $ M\upharpoonright W $ admits an  $ N $-independent base $ B $. 
  Then $ N/W $ can be written as the direct sum of strictly less than $ n $ uniform matroids. We apply the induction 
  hypothesis with  $ M/W $ and $ N/W $ to obtain  a witness $ I_{M/W}\sqcup I_{N/W}$ showing that $ \{M/W, N/W \} $ has 
  the Intersection property. We define  $ I_M:=I_{M/W}\cup B $ and $ I_N:= I_{N/W}$.   Then 
     \begin{gather*}
     I_M\sqcup I_N\in \mathcal{I}_M\cap \mathcal{I}_N\\ \text{ because }\\  I_{M/W}\sqcup I_{N/W}\in 
    \mathcal{I}_{M/W}\cap 
    \mathcal{I}_{N/W} \text{ and }\\ B\in \mathcal{I}_{M \upharpoonright W}\cap \mathcal{I}_{N 
    \upharpoonright W}. 
     \end{gather*}
    Since $ I_{M/W}\sqcup I_{N/W}$ shows the Intersection property of $ \{M/W, N/W \} $, set  $ 
    I_{M/W} $ spans  \[ E\setminus (W \cup  \mathsf{span}_{N/W}(I_{N/W}))  \] in $ M/W $.
    
    \begin{obs}\label{obs: N-W=N/W}
      $ N-W=N/W $ since
        \[ N= (N\upharpoonright W)\oplus (N-W). \] 
      \end{obs}
    By $I_N=I_{N/W}  $ and by Observation \ref{obs: N-W=N/W}  we obtain  \[ 
   \mathsf{span}_{N/W}(I_{N/W})=\mathsf{span}_{N/W}(I_{N})=\mathsf{span}_{N-W}(I_{N})=\mathsf{span}_N(I_{N}). 
   \] Thus $ 
       I_{M/W} $ spans actually  \[ E\setminus (W\cup  \mathsf{span}_N(I_{N}))  \] in $ M/W $.
     But then,  since $ B $ is a base of $ M \upharpoonright W $, $ I_{M}=I_{M/W}\cup B $ spans $  E \setminus  
     \mathsf{span}_N(I_{N})  $ in $ M $. Therefore $ I_M\sqcup I_N $ witnesses the Intersection property of $ \{ M, N \} $.

  Suppose now that there is no such a $ W $, i.e., the following holds:
 \begin{cond}\label{cond}
 There is no nonempty $W\subseteq E $ which is the union of some of the sets $ E_i $  such that 
 $ M\upharpoonright W $ admits an  $ N $-independent base $ B $.
 \end{cond}
 
 In this case we can  finish the proof by applying the following theorem with $ J=\varnothing $:
 \begin{thm}\label{thm:side}
 Let $ M $ and $ N $ be matroids on  $ E $ such that  $ N=\bigoplus_{i<n}U_i $ where $ n\in 
 \mathbb{N}$ and  $ U_i $ is a uniform 
 matroid on  $ E_i $  with $ E_i\cap E_j=\varnothing $ for $ 0\leq i<j<n $. Assume that Condition 
 \ref{cond} holds. Then for every $ J\in \mathcal{I}_M\cap 
 \mathcal{I}_N $ there 
 exists an $ M $-independent base $ B $ of $ N $ with $J\subseteq \mathsf{span}_M(B) $.
 \end{thm}
 Indeed, a $ B $ provided by 
 Theorem \ref{thm:side} for 
 $ J=\varnothing $  witnesses the Intersection property of $ \{ M, N \} $ by Observation 
 \ref{obs:baseContained}. 

\subsection{Proof of Theorem \ref{thm:side}}

We use induction on $ n $. For $ n=0 $, the matroids are trivial, we must have $ J=\varnothing $ and $ \varnothing $ is a desired $ 
M $-independent base of $ N $. Suppose that $ n\geq1 $.  For 
$ I\in 
\mathcal{I}_M\cap 
\mathcal{I}_N $, let  \[ \Theta(I):=\{ i<n: E_i\subseteq \mathsf{span}_M(I)  \}.  \]
We take an $ I\in \mathcal{I}_M\cap \mathcal{I}_N $ with $  \mathsf{span}_M(I) \supseteq J $ that maximizes $ \Theta(I) $ in 
the sense that $\Theta(K)=\Theta(I)  $ holds for every  $  K\in \mathcal{I}_M\cap \mathcal{I}_N $ with $  \mathsf{span}_M(K) 
\supseteq I $. Note that $\Theta(I)\subsetneq n$ since otherwise $ W:=E $ would violate condition \ref{cond}.
 By symmetry we can assume without loss of generality that $ \Theta(I)=\{ 0,\dots, k-1 
 \} $ for some $ k<n $. We consider 
  $ E':=\bigcup_{i<k}E_i$ and $ J':=I\cap E' $ together with matroids $ M':=M\upharpoonright E' $ and $ N':=N\upharpoonright 
  E' $. 
  Then  $ N'=\bigoplus_{i<k}U_i $ and Condition \ref{cond} is satisfied by $ M' $ and $ N' $ because a violating $ 
  W $ would be also a violation with respect to $ M $ and $ N $. By induction we get an $ M' $-independent base $ B' $ of $ N' $ 
  with  $J'\subseteq \mathsf{span}_{M'}(B') $. Clearly
 \[  \begin{gathered}
   (I\setminus E')\cup B'\in \mathcal{I}_N\\ 
   \text{ because}\\ 
   B'\in \mathcal{I}_{N \upharpoonright E'},  (I\setminus E')\in \mathcal{I}_{N\upharpoonright (E\setminus E')}\\
   \text{ and }\\
    N=(N\upharpoonright E') \oplus (N-E').
  \end{gathered} \]

    We also know that $ (I\setminus E')\cup B' $ spans $ I $ in $ M $ 
  because $ I\cap E'=J'  $  is spanned by $ B' $ in $ M $.  We  extend $ B' $ to a maximal $ M $-independent subset of $ 
  (I\setminus E')\cup B' $ and then further extend the resulting set via Lemma \ref{lem:extend} 
  to a maximal element $ B $ 
  of $ \mathcal{I}_M\cap \mathcal{I}_N $. Then $ \mathsf{span}_M(B)\supseteq I $ by construction which implies that $  
  \mathsf{span}_M(B)\supseteq J  $ and $ 
  \Theta(B)=\Theta(I) $. On the one hand,  it follows from  $  \Theta(B)=\Theta(I) $ by
  Lemma \ref{lem:char maximal} via the maximality of $ B $ that $E_i\subseteq \mathsf{span}_N(B) $
    for every $ k\leq i<n $. On the 
  other hand, $ B'\subseteq B $ is a base of  $N'= \bigoplus_{i<k}U_i $. By combining these, it
 follows that $E_i\subseteq \mathsf{span}_N(B) $ for every $ i<n $. Thus $ B $ is an $ M $-independent base of $ N $ with 
$  \mathsf{span}_M(B)\supseteq J $ and therefore the proof of Theorem \ref{thm:side} is complete.

\end{document}